\documentclass{article}
\usepackage{preamble}
\title{Gray-categories model algebraic tricategories}
\author{Giovanni Ferrer}

\begin{document}

\maketitle

\begin{abstract}
Lack described a Quillen model structure on the category $\GrayCat$ of $\Gray$-categories and $\Gray$-functors, for which the weak equivalences are the weak 3-equivalences. 
In this note, we adapt the technique of Gurski, Johnson, and Osorno to show the localization of $\GrayCat$ at the weak equivalences is equivalent to the category of algebraic tricategories and pseudo-natural equivalence classes of weak 3-functors.
\end{abstract}

\tableofcontents
\section{Introduction}

Let $\TwoCat$ be the category of strict 2-categories and strict 2-functors. 
The symmetric monoidal category $\Gray$ is the category $\TwoCat$ equipped with the Gray monoidal structure \cite{gurski_2013}. 
A $\Gray$-category is then a category enriched in $\Gray$ in the sense of \cite{kelly_2005}. 
We denote the 1-category of $\Gray$-categories and $\Gray$-functors by $\GrayCat$. A more explicit definition of $\Gray$-category can be found in Appendix \ref{appendix:Gray}.\smallskip

We now discuss the model category structure on $\GrayCat$ from \cite{lack_2011}. 

\begin{definition}
A $\Gray$-functor $F: \cA \to \cB$ between $\Gray$-categories is called:
\begin{itemize}
\item a \emph{weak equivalence} if it is a weak 3-equivalence of $\cA$ and $\cB$ as weak 3-categories,

\item a \emph{fibration} if it is a fibration on each hom-2-category and satisfies a so-called biadjoint-biequivalence-lifting property (from which one sees all $\Gray$-categories are fibrant),

\item in particular, a \emph{trivial fibration} if it is surjective at all levels and (fully) faithful at the top level, and
    
\item a \emph{cofibration} if it has the left lifting property against trivial fibrations. 
More specifically, $F: \cA \to \cB$ is a cofibration if and only if for every trivial fibration $F': \cA' \to \cB'$ and $\Gray$-functors $a: \cA \to \cA'$ and $b: \cB \to \cB'$ such that the following square commutes,
    
\begin{center}
\begin{tikzcd}
\cA \arrow[r,"a"] \arrow[d,"F"'] & \cA' \arrow[d,"F'"]\\
\cB \arrow[r,"b"'] \arrow[ur,"\ell" description, dashed] & \cB'
\end{tikzcd}
\end{center}
there exists a $\Gray$-functor lift $\ell:\cB \to \cA'$ such that the two triangles commute.
\end{itemize}
\end{definition}

\begin{definition}
We denote by $\cW$ the collection of weak equivalences in $\GrayCat$. The category $\GrayCat[\cW^{-1}]$ is the localization of $\GrayCat$ at the weak equivalences, which is determined (up to unique isomorphism) by the following universal property \cite{bauer_dugundji_1969}:
\begin{itemize}
    \item For any functor $F:\GrayCat \to \cD$ of categories which maps weak equivalences to isomorphisms, $F$ uniquely factors through $\GrayCat[\cW^{-1}]$, i.e., there exists a unique $\Phi: \GrayCat[\cW^{-1}] \to \cD$ such that the following diagram commutes on the nose:
\vspace{-10pt}
\begin{center}
\begin{tikzcd}[column sep = 4pt, row sep = 7pt]
\GrayCat \arrow[dd, "\pi"'] \arrow[rr, "F"] &    & \cD \\
                                                                                 & {} &     \\
{\GrayCat[\cW^{-1}]} \arrow[rruu,dashed, "\exists!\, \Phi"']                              &    &    
\end{tikzcd}
\end{center}
\end{itemize}
\end{definition}
\begin{remark}
In \cite[Appendix]{bauer_dugundji_1969}, an explicit construction is provided in which an object of $\GrayCat[\cW^{-1}]$ is simply a $\Gray$-category and a morphism from $\cA \to \cB$ in $\GrayCat[\cW^{-1}]$ is a ``zigzag'', i.e., a finite chain of $\Gray$-categories and $\Gray$-functors
$$\cA \xleftarrow{w_1} X_1 \xrightarrow{f_1} X_2 \xleftarrow{w_2} X_3 \xrightarrow{f_2} \cdots \xleftarrow{w_n} X_{2n-1} \xrightarrow{f_n} \cB$$
where the $\leftarrow$ and $\rightarrow$ alternate and each morphism $w_i$ ``in the wrong direction'' is a weak equivalence. These zigzags are considered equal up to a certain equivalence relation. We refer the reader to the previous citation for more details.
\end{remark}
\begin{definition}
We denote by $\hoTriCat$ the 1-category whose objects are $\Gray$-categories and morphisms are pseudo-natural equivalence classes of weak 3-functors. Since each weak 3-category (algebraic tricategory) is equivalent a $\Gray$-category \cite[\S10.4]{gurski_2013}, $\hoTriCat$ is equivalent to the 1-category of weak 3-categories and pseudo-natural equivalence classes of weak 3-functors.
\end{definition}
The main result of this note shows that Lack's model structure models algebraic tricategories in the following sense.
\begin{theorem*}[Theorem \ref{thm:main}]
The categories $\GrayCat[\cW^{-1}]$ and $\hoTriCat$ are isomorphic.
\end{theorem*}
The proof is a straightforward adaptation of \cite[Prop.~3.31]{gurski_johnson_osorno_2019} using cofibrant replacement and path objects.
\begin{acknowledgements}
The author would like to thank David Penneys for his guidance in this project and for many helpful conversations. The author would also like to thank Nick Gurski, Niles Johnson, and David Reutter for useful suggestions.
\end{acknowledgements}

\section{Constructions for \texorpdfstring{$\Gray$}{Gray}-categories}
In this section we provide the two necessary constructions needed to adapt the technique of \cite[Prop.~3.31]{gurski_johnson_osorno_2019} to show Theorem \ref{thm:main}.

\begin{prop}[Cofibrant replacement in $\GrayCat$] \label{fact:hat}
For every $\cC \in \GrayCat$, there is a cofibrant $\widehat{\cC} \in \GrayCat$ together with an ``evaluation'' $\Gray$-functor $\ev_{\cC}: \widehat{\cC} \to \cC$ which is a trivial fibration.
For every $\cA,\cB \in \GrayCat$ and weak 3-functor $F: \cA \to \cB$, there is a $\Gray$-functor $\widehat{F}: \widehat{\cA} \to \widehat{\cB}$ which satisfies the following properties:
\begin{enumerate}
\item[$(\widehat{1})$] For every weak 3-functor $F: \cA \to \cB$, the diagram \begin{tikzcd}
\widehat{\cA} \arrow[d, "\ev_{\cA}"'] \arrow[r, "\widehat{F}"] & \widehat{\cB} \arrow[d, "\ev_{\cB}"] \\
\cA \arrow[r, "F"']                                & \cB                           
\end{tikzcd} weakly commutes.

\item[$(\widehat{2})$] When $F: \cA \to \cB$ is a $\Gray$-functor, the diagram in $(\widehat{1})$ strictly commutes.

\item[$(\widehat{3})$] $\widehat{ \id_{\cA}} = \id_{\widehat{ \cA}}$.

\item[$(\widehat{4})$] If $F: \cA \to \cB$ and $G: \cB \to \cC$ are weak 3-functors, then $\widehat{ G \circ F} \cong \widehat{ G } \circ \widehat{ F }$.

\item[$(\widehat{5})$] In $(\widehat{4})$, if either $F$ or $G$ is a $\Gray$-functor, then $\widehat{ G \circ F} = \widehat{ G } \circ \widehat{ F }$.
\end{enumerate}
\end{prop}

\begin{proof}
Gurski's $\Gr$ construction in \cite[\S10.4]{gurski_2013} satisfies the properties required of $\widehat{\,\cdot\,}$.
We recall this construction and go over its desired properties in Appendix \ref{appendix:Gr}.
\end{proof}

\begin{prop}[Path objects exist in $\GrayCat$]\label{fact:BI}
For every $\cB \in \GrayCat$, there exists a ``path object'' $\cB^I \in \GrayCat$ such that:
\begin{enumerate}
    \item[$\mathrm{(P1)}$] There exist $\Gray$-functors $
        \begin{tikzcd}
\cB \arrow[r, "C" description] & \cB^I \arrow[r, "T" description, bend right] \arrow[r, "S" description, bend left] & \cB
\end{tikzcd}$
where $C$ is a weak equivalence and $S \circ C = T \circ C =\id_\cB$.

Here $\binom{S}{T}: \cB^I \to \cB \times \cB$ is a fibration, so that $\cB^I$ is a path object in the sense of \cite[Ch.~1, Def.~4]{quillen_1967}.

\item[$\mathrm{(P2)}$] If $F: \cB_1 \to \cB_2$ is a $\Gray$-functor, then there exists a $\Gray$-functor $F^I: \cB_1^I \to \cB_2^I$ which makes the corresponding $C$, $S$, and $T$ squares commute:
\vspace{-10pt}
\begin{center}
\begin{tikzcd}
\cB_1 \arrow[r, "C" description] \arrow[d, "F"'] & \cB_1^I \arrow[d, "F^I"'] \arrow[r, "S" description, bend left] \arrow[r, "T" description, bend right] & \cB_1 \arrow[d, "F"] \\
\cB_2 \arrow[r, "C" description]                 & \cB_2^I \arrow[r, "S" description, bend left] \arrow[r, "T" description, bend right]                   & \cB_2               
\end{tikzcd}  
\end{center}
    \item[$\mathrm{(P3)}$] Moreover, $\cB^I$ satisfies the property that for every two pseudo-naturally equivalent $\Gray$-functors $ F , G : \cA \to \cB$, there is a weak 3-functor $\langle F ,G \rangle: \cA \to \cB^I$ such that $S \circ \langle F , G  \rangle = F  $ and $T \circ \langle F  ,G  \rangle = G $.
\end{enumerate}
\end{prop}

\begin{proof}
Lack's path space construction $\bbP\cB$ for $\cB \in \GrayCat$ in \cite[Prop.~4.1]{lack_2011} satisfies the desired properties. As \cite{gurski_johnson_osorno_2019} note in their Remark 3.33, there have been some mistakes in the literature regarding path objects and transferred model structures. 
We provide a careful treatment of this by recalling Lack's construction and showing it satisfies the properties of $\cB^I$ in Appendix \ref{appendix:PB}.
\end{proof}

\begin{remark}
We note that the existence of such path objects in $\GrayCat$ satisfying \hyperref[fact:BI]{(P3)} is crucial in proving our main result, as it relates (right) homotopies in the model theoretic sense with pseudo-natural equivalences. 
\end{remark}

\section{The main theorem}
In this section we implement the technique of \cite{gurski_johnson_osorno_2019}. There is an obvious functor $Q: \GrayCat\to \hoTriCat$ which maps each $\Gray$-category to itself and maps $F: \cA \to \cB$ to its equivalence class $[F]$ in $\hoTriCat$.
Clearly this functor maps equivalences to isomorphisms, and thus uniquely factors through $\GrayCat[\cW^{-1}]$.
We will denote this factorization by $\Phi: \GrayCat[\cW^{-1}] \to \hoTriCat$.\medskip

\begin{definition}
Define $\Psi: \hoTriCat \to \GrayCat[\cW^{-1}]$ to be the identity on objects and, for a weak 3-functor $F: \cA \to \cB$, define
$$\Psi([F]) \coloneqq \cA \xleftarrow{\ev_{\cA}} \widehat{\cA} \xrightarrow{\widehat{F}} \widehat{\cB} \xrightarrow{\ev_{\cB}} \cB.$$
\end{definition}

\begin{lemma}
$\Psi$ is a well-defined functor.
\end{lemma}
\begin{proof}
Suppose $F,G: \cA \to \cB$ are pseudo-naturally equivalent. 
Then
\begin{align*}
\Psi([F]) 
&= \cA \xleftarrow{\ev_{\cA}} \widehat{\cA} \xrightarrow{\widehat{F}} \widehat{\cB} \xrightarrow{\ev_{\cB}} \cB\\
&= \cA \xleftarrow{\ev_{\cA}} \widehat{\cA} \xrightarrow{\widehat{S \circ \langle F,G \rangle}} \widehat{\cB} \xrightarrow{\ev_{\cB}} \cB \tag*{\hyperref[fact:BI]{(P3)}}\\
&= \cA \xleftarrow{\ev_{\cA}} \widehat{\cA} \xrightarrow{\widehat{ \langle F,G \rangle}} \widehat{\cB^I} \xrightarrow{\widehat{S}} \widehat{\cB} \xrightarrow{\ev_{\cB}} \cB \tag*{\hyperref[fact:hat]{($\widehat{5}$)}}\\
&= \cA \xleftarrow{\ev_{\cA}} \widehat{\cA} \xrightarrow{\widehat{ \langle F,G \rangle}} \widehat{\cB^I} \xrightarrow{\ev_{\cB^I}} \cB^I \xrightarrow{S} \cB \tag*{\hyperref[fact:hat]{($\widehat{2}$)}}\\
&= \cA \xleftarrow{\ev_{\cA}} \widehat{\cA} \xrightarrow{\widehat{ \langle F,G \rangle}} \widehat{\cB^I} \xrightarrow{\ev_{\cB^I}} \cB^I \xleftarrow{C} \cB \xrightarrow{C} \cB^I \xrightarrow{S} \cB\\
&= \cA \xleftarrow{\ev_{\cA}} \widehat{\cA} \xrightarrow{\widehat{ \langle F,G \rangle}} \widehat{\cB^I} \xrightarrow{\ev_{\cB^I}} \cB^I \xleftarrow{C} \cB \xrightarrow{C} \cB^I \xrightarrow{T} \cB \tag*{\hyperref[fact:BI]{(P1)}}\\
&= \cA \xleftarrow{\ev_{\cA}} \widehat{\cA} \xrightarrow{\widehat{ \langle F,G \rangle}} \widehat{\cB^I} \xrightarrow{\ev_{\cB^I}} \cB^I \xrightarrow{T} \cB \\
&= \cA \xleftarrow{\ev_{\cA}} \widehat{\cA} \xrightarrow{\widehat{ \langle F,G \rangle}} \widehat{\cB^I} \xrightarrow{\widehat{T}} \widehat{\cB} \xrightarrow{\ev_{\cB}} \cB \tag*{\hyperref[fact:hat]{($\widehat{2}$)}}\\
&= \cA \xleftarrow{\ev_{\cA}} \widehat{\cA} \xrightarrow{\widehat{T \circ \langle F,G \rangle}} \widehat{\cB} \xrightarrow{\ev_{\cB}} \cB \tag*{\hyperref[fact:hat]{($\widehat{5}$)}}\\
&= \cA \xleftarrow{\ev_{\cA}} \widehat{\cA} \xrightarrow{\widehat{G}} \widehat{\cB} \xrightarrow{\ev_{\cB}} \cB  = 
\Psi([G]). \tag*{\hyperref[fact:BI]{(P3)}}
\end{align*}
Hence $\Psi$ is well-defined. Observe
\begin{align*}
\Psi([\id_\cA]) 
&= \cA \xleftarrow{\ev_{
\cA}} \widehat{\cA} \xrightarrow{\widehat{\id_{\cA}}} \widehat{\cA} \xrightarrow{\ev_{\cA}} \cA \\
&= \cA \xleftarrow{\ev_{
\cA}} \widehat{\cA} \xrightarrow{\id_{\widehat{{\cA}}}} \widehat{\cA} \xrightarrow{\ev_{\cA}} \cA \tag*{\hyperref[fact:hat]{($\widehat{3}$)}}\\
&= \cA \xleftarrow{\ev_{
\cA}} \widehat{\cA} \xrightarrow{\ev_{\cA}} \cA\\
&= \cA \xrightarrow{\id_{\cA}} \cA = \id_{\Psi(\cA)}.
\end{align*}
Since $\Psi$ is well-defined, for composable weak 3-functors $\cA \xrightarrow{F} \cB \xrightarrow{G} \cC$,
\begin{align*}
\widehat{\cA} \xrightarrow{\widehat{F}} \widehat{\cB} \xrightarrow{\widehat{G}} \widehat{\cC} \underset{\hyperref[fact:hat]{(\widehat{2})}}{=} 
\widehat{\cA} \xleftarrow{\ev_{\widehat{\cA}}} \,\,\widehat{\!\!\widehat{{\cA}}} \xrightarrow{\widehat{\widehat{G} \circ \widehat{F}}} \,\,\widehat{\!\!\widehat{{\cC}}} \xrightarrow{\ev_{\widehat{\cC}}} \widehat{\cC} 
&= \Psi([\widehat{G} \circ\widehat{F}]) \tag*{\hyperref[fact:hat]{($\widehat{4}$)}}\\[-10pt]
&= \Psi([\widehat{G \circ F}]) =
\widehat{\cA} \xleftarrow{\ev_{\widehat{\cA}}} \,\,\widehat{\!\!\widehat{{\cA}}} \xrightarrow{\widehat{\widehat{G \circ F}}} \,\,\widehat{\!\!\widehat{{\cC}}} \xrightarrow{\ev_{\widehat{\cC}}} \widehat{\cC}
\underset{\hyperref[fact:hat]{(\widehat{2})}}{=} 
\widehat{\cA} \xrightarrow{\widehat{G \circ F}} \widehat{\cC}.
\end{align*}
\vspace{-7pt} Hence
\begin{align*}
\Psi([G]) \circ \Psi([F]) 
&= \cA \xleftarrow{\ev_{\cA}} \widehat{\cA} \xrightarrow{\widehat{F}} \widehat{\cB} \xrightarrow{\ev_{\cB}} \cB \xleftarrow{\ev_{\cB}} \widehat{\cB} \xrightarrow{\widehat{G}} \widehat{\cC} \xrightarrow{\ev_{\cC}} \cC\\
&= \cA \xleftarrow{\ev_{\cA}} \widehat{\cA} \xrightarrow{\widehat{F}} \widehat{\cB} \xrightarrow{\widehat{G}} \widehat{\cC} \xrightarrow{\ev_{\cC}} \cC \\    
&= \cA \xleftarrow{\ev_{\cA}} \widehat{\cA} \xrightarrow{\widehat{G \circ F}} \widehat{\cC} \xrightarrow{\ev_{\cC}} \cC = 
\Psi([G 
\circ F]),\\  
\end{align*}
and we conclude $\Psi$ is a functor.
\end{proof}

We are now ready to prove the main result of this note. 

\begin{theorem}\label{thm:main}
The functors $\Phi$ and $\Psi$ exhibit an isomorphism of categories.
\end{theorem}
\begin{proof}
Since both functors are the identity on objects, it suffices to show that $\Phi$ and $\Psi$ are mutual inverses on hom sets.
The fact that $\Phi$ is surjective on hom sets follows from \hyperref[fact:hat]{($\widehat{1}$)}.
We now show $(\Psi \circ \Phi)([F]) = [F]$ and conclude $\Phi$ is an isomorphism with inverse $\Psi$. 
Indeed, consider the following diagram:

\begin{center}
\begin{tikzcd}[column sep = 5pt, row sep = 15pt]
\GrayCat \arrow[dd, "\pi"'] \arrow[rr, "Q"] &    & \hoTriCat \arrow[dd, "\Psi"] \\
                                                                                 & {} &                              \\
{\GrayCat[\cW^{-1}]} \arrow[rr, "\id"'] \arrow[rruu, "\Phi" description]         &    & {\GrayCat[\cW^{-1}]}        
\end{tikzcd}
\end{center}
whose left triangle commutes by the universal property of $\GrayCat[\cW^{-1}]$. 
Note that for every $\Gray$-functor $F: \cA \to \cB$,
\begin{equation*}
(\Psi \circ Q)(F) = \Psi([F]) = A \xleftarrow{\ev_{\cA}} \widehat{\cA} \xrightarrow{\widehat{F}} \widehat{\cB} \xrightarrow{\ev_{\cB}} \cB \underset{\hyperref[fact:hat]{(\widehat{2})}}{=} \cA \xleftarrow{\ev_{\cA}} \cA \xrightarrow{\ev_{\cA}} \cA \xrightarrow{F} \cB = \cA \xrightarrow{F} \cB = \pi(F). 
\end{equation*}
Thus $\Psi \circ \Phi \circ \pi  = \Psi \circ Q = \pi$. By the uniqueness of factorizations through $\GrayCat[\cW^{-1}]$, it is easy to see that $\pi$ is epic.
From this we conclude $\Psi \circ \Phi = \id$.
\end{proof}

\section{Corollaries}

We obtain the following two corollaries which correspond to the surjectivity and injectivity/well-definedness respectively of the above bijection induced by $\Phi$.

\begin{corollary}\label{cor:cofibweakfuncts}
If $\cA$ is a cofibrant $\Gray$-category and $F: \cA \to \cB$ is a weak 3-functor, then $F$ is pseudonaturally equivalent to a $\Gray$-functor.
\end{corollary}

\begin{proof}
Since $\ev_{\cA}: \widehat{\cA} \to \cA$ is a trivial fibration, if $\varnothing \to \cA$ is a cofibration, then there exists a $\Gray$-functor lift $\ell: \cA \to \widehat{\cA}$ which makes the following triangles commute
\begin{center}
\begin{tikzcd}
\varnothing \arrow[d, "!"'] \arrow[r, "!"]                   & \widehat{\cA} \arrow[d, "\ev_{\cA}"] \\
\cA \arrow[r, "\id"'] \arrow[ru, "\ell" description, dotted] & \cA                                 
\end{tikzcd}
\end{center}
Thus $\ell$ is a section of $\ev_\cA$ and $
F = F \circ \ev_\cA \circ \, \ell \underset{\hyperref[fact:hat]{(\widehat{1})}}{\cong} \ev_{\cB} \circ \widehat{F} \circ \ell \text{, a $\Gray$-functor}.$
\end{proof}

\begin{corollary}
Suppose $\cA$ is a cofibrant $\Gray$-category and $F,G: \cA \to \cB$ are $\Gray$-functors. Then $F$ is pseudonaturally equivalent to $G$ if and only if $F$ and $G$ are homotopic in Lack's model structure.
\end{corollary}

\begin{proof}
First suppose $F$ is pseudo-natually equivalent to $G$, i.e. $\Phi(F) = [F] = [G] = \Phi(G)$. By the injectivity of $\Phi$, $F = G$ in $\GrayCat[\cW^{-1}]$. Since $\ev_\cA, \ev_\cB \in \cW$, it follows by \hyperref[fact:hat]{($\widehat{2}$)} that $\widehat{F} = \widehat{G}$ in $\GrayCat[\cW^{-1}]$. By the general theory of model categories, $\GrayCat[\cW^{-1}]$ is equivalent to the category of fibrant-cofibrant $\Gray$-categories together with homotopy classes of $\Gray$-functors (see \cite[Thm.~8.3.9]{hirschhorn_2003}). Through this equivalence we obtain that $\widehat{F}$ is homotopic to $\widehat{G}$ since $\widehat{A}$ and $\widehat{B}$ are fibrant-cofibrant. More specifically, there exists a $\Gray$-functor $H: \widehat{A} \rightarrow \widehat{B}^I$ such that $S \circ H = \widehat{F}$ and $T \circ H = \widehat{G}$.

By \hyperref[fact:BI]{(P2)}, there exists a map $\ev_\cB^I: \widehat{B}^I \to \cB^I$ such that the following diagrams commute.
\begin{center}
\begin{tikzcd}
\widehat{\cB}^I \arrow[r, "S"] \arrow[d, "\ev_\cB^I"'] & \widehat{\cB} \arrow[d, "\ev_\cB"] \\
\cB^I \arrow[r, "S"']                                  & B                                 
\end{tikzcd}    
\hspace{50pt}
\begin{tikzcd}
\widehat{\cB}^I \arrow[r, "T"] \arrow[d, "\ev_\cB^I"'] & \widehat{\cB} \arrow[d, "\ev_\cB"] \\
\cB^I \arrow[r, "T"']                                  & B                                 
\end{tikzcd}
\end{center}
As before, since $\cA$ is cofibrant, there exists a $\Gray$-functor section $\ell: \cA \to \widehat{\cA}$ for $\ev_\cA$. By defining the $\Gray$-functor $H': \cA \to \cB^I$ by $H' \coloneqq \ev^I_\cB \circ H \circ \ell$ we see that 
\begin{align*}
S \circ H' &= \cA \xrightarrow{\ell} \widehat{\cA} \xrightarrow{H} \widehat{\cB}^I \xrightarrow{\ev_\cB^I} \cB^I \xrightarrow{S} \cB\\
&= \cA \xrightarrow{\ell} \widehat{\cA} \xrightarrow{H} \widehat{\cB}^I \xrightarrow{S} \widehat{\cB} \xrightarrow{\ev_\cB} \cB\\
&= \cA \xrightarrow{\ell} \widehat{\cA} \xrightarrow{\widehat{F}} \widehat{\cB} \xrightarrow{\ev_\cB} \cB\\
&\underset{\hyperref[fact:hat]{(\widehat{2})}}{=}  \cA \xrightarrow{\ell} \widehat{\cA} \xrightarrow{\ev_\cA} \cA \xrightarrow{F} \cB = F
\end{align*}
and similarly $T \circ H' = G$. Thus $F,G$ are homotopic. Conversely, suppose $F,G$ are homotopic so there exists a $\Gray$-functor $H: \cA \to \cB^I$ with $S \circ H = F$ and $T \circ H = G$. By \hyperref[fact:BI]{(P1)},  $(C: \cB \to \cB^I) \in \cW$ with $S \circ C = T \circ C$, so $S = T$ in $\GrayCat[\cW^{-1}].$ This implies $F = S \circ H = T \circ H = G$ in $\GrayCat[\cW^{-1}]$ and by the well-definedness of $\Phi$ we have that $[F] = \Phi(F) = \Phi(G) = [G]$. Thus $F$ is pseudo-naturally equivalent to $G$.
\end{proof}

\appendix
\section{\texorpdfstring{$\Gray$}{Gray}-categories}\label{appendix:Gray}
We present the data and axioms of a $\Gray$-category in order to fix notation for Appendix \ref{appendix:Gr} and \ref{appendix:PB}.
\begin{notation}
\label{rem:UnpackGrayMonoid}
A $\Gray$-category $\cC$ consists of the following data:
\begin{enumerate}[label=(D\arabic*)]
\item[(D0)] 
\label{Gray:Objects}
a collection of objects $\cC_0$, denoted by lower case letters $a,b,c$ 
\item 
\label{Gray:2cat}
for $a,b \in \cC_0$, a strict $2$-category $\cC(a,b)$ where we write $f: a \to b$ whenever $f \in \cC(a,b)$, composition of 1-morphisms (called 2-morphisms in $\cC$) is denoted by $\otimes$, and composition of 2-morphisms (called 3-morphisms in $\cC$) is denoted by $\circ$; 
\item 
\label{Gray:Id}
for each $a \in \cC_0$, an \emph{identity} $\id_a: a \to a$;
\item 
\label{Gray:tensor}
for objects $a,b,c,d \in \cC_0$ and 1-morphisms $g:c \to d$, $f: a \to b$, strict covariant and contravariant \emph{hom$-2$-functors} $g_* = g\boxtimes -$ and $f^* = -\boxtimes f$:
\begin{align*}
g_* : \cC(b,c) \to \cC(b,d) \quad \text{and} \quad
f^* : \cC(b,c) \to \cC(a,c), 
\end{align*}
\item
\label{Gray:Interchanger}
an \emph{interchanger} $3$-isomorphism
$\Sigma_{\gamma,\xi}$ for each pair of ``horizontally composable'' 2-morphisms $\xi: g \Rightarrow g'$ with $g,g': b \to c$ and $\gamma:f \Rightarrow  f'$ with $f,f':a \to b$:
\[
\Sigma_{\xi,\gamma}: \left( \xi \boxtimes \id_{f'}\right)\otimes \left(\id_{g} \boxtimes \gamma\right) \Rrightarrow \left(\id_{g'} \boxtimes \gamma \right)\otimes \left(\xi \boxtimes \id_f \right)
\]
\end{enumerate}
subject to the following conditions:
\begin{enumerate}[label=(C\arabic*)]
\item 
pre-$\boxtimes$ and post-$\boxtimes$ agree, i.e., 
for composable 1-morphisms $g:b \to c$ and $f : a \to b$, 
$$g_* f = f^* g = g \boxtimes f;$$
\item 
$\boxtimes$ is strictly unital and associative, i.e., the following hold whenever they make sense:
\begin{align*} 
(\id_b)_* &= \mathrm{id}_{\cC(a,b)} = (\id_a)^* \\
g_* f_* &= (g\boxtimes f)_*\\
f^*  g^* &= (g\boxtimes f)^* \\
g_* f^* &= f^* g_*;
\end{align*}
\item the interchanger $\Sigma$ respects identities, i.e., 
for a 1-morphism $f: b \to c$ and 2-morphisms $\xi,\gamma$, the following hold whenever they make sense:
\begin{align*}
\Sigma_{\xi, \id_f} = \id_{\xi \boxtimes f} \qquad \text{and} \qquad
\Sigma_{\id_f, \gamma} = \id_{f \boxtimes \gamma}
\end{align*}
\item
\label{Interchanger:Composition}
the interchanger $\Sigma$ respects $\otimes$, i.e., 
for $g \xRightarrow{\xi} g'\xRightarrow{\xi'} g''$ and $f \xRightarrow{\gamma} f'\xRightarrow{\gamma'} f''$, the following hold whenever they make sense:
\begin{align*}
\Sigma_{\xi'\otimes \xi, \gamma}&= \left(\Sigma_{\xi',\gamma} \otimes (\xi\boxtimes f ) \right)\circ \left( (\xi' \boxtimes {f'})\otimes \Sigma_{\xi,\gamma}\right)\\
\Sigma_{\xi,\gamma'\otimes \gamma} &= \left(({g'} \boxtimes\gamma')  \otimes \Sigma_{\xi,\gamma}\right) \circ \left(\Sigma_{\xi,\gamma'}\otimes ( g \boxtimes \gamma )  \right)
\end{align*}
\item 
\label{Interchanger:Natural}
the interchanger $\Sigma$ is natural, i.e., 
for $g,g': b \to c$, $\xi,\xi':g\Rightarrow g'$ and $\Xi: \xi\Rrightarrow \xi'$; and $f,f': a \to b$, $\gamma,\gamma':f\Rightarrow f'$ and $\Gamma:\gamma\Rrightarrow \gamma'$:
\begin{align*}
\Sigma_{\xi',\gamma} \circ \left((\Xi \boxtimes {f'}) \otimes \id_{{g}\boxtimes \gamma} \right) &= \left(\id_{{g'}\boxtimes \gamma} \otimes (\Xi \boxtimes f) \right)\circ \Sigma_{\xi,\gamma}\\
\Sigma_{\xi,\gamma'}\circ \left( \id_{\xi\boxtimes {f'}} \otimes \left(g \boxtimes  \Gamma \right) \right) &=  \left( (g' \boxtimes \Gamma ) \otimes \id_{ \xi \boxtimes f } \right) \circ\Sigma_{\xi,\gamma}
\end{align*}
\item 
the interchanger $\Sigma$ respects $\boxtimes$, i.e., 
for $1$-morphisms $f,g,h$ and 2-morphisms $\sigma,\xi,\gamma$, the following hold whenever they make sense:
\begin{align*}
\Sigma_{ h\boxtimes \xi, \gamma} = h \boxtimes \Sigma_{\xi,\gamma} &&
\Sigma_{\sigma\boxtimes g, \gamma} = \Sigma_{\sigma, g\boxtimes \gamma} &&
\Sigma_{\sigma,\xi\boxtimes f} = \Sigma_{\sigma,\xi} \boxtimes f
\end{align*}
\end{enumerate}
\end{notation}

\section{Gurski's \texorpdfstring{$\Gr$}{Gr} construction}\label{appendix:Gr}
We recall \cite[Def.~10.7]{gurski_2013} specialized to the case of a $\Gray$-category.
\begin{construction}\label{construction:Gr}
The $\Gray$-category $\Gr \cA$ of a $\Gray$-category $\cA$ is constructed as follows. 
\begin{itemize}
    \item[($\Gr$0)] $\Gr \cA$ has the same objects as $\cA$, i.e. $(\Gr \cA)_0 = \cA_0$. 

    \item[($\Gr$1)] For $a,b \in \cA_0$, the objects in the 2-category $(\Gr \cA) (a,b)$ are (finite) strings $\{f_i\}$ of composable 1-cells of $\cA$ (starting at $a$ and ending at $b$). For composable lists $
    \{g_j\}$ and $\{f_i\}$ we define $\{g_j\} \boxtimes \{f_i\}$ to be their concatenation $\{f_i, g_j\}$. Notice the identity for an object $a \in \Gr \cA$ will be the empty string $\varnothing_a$ starting and ending in $a$.

\item[($\Gr$2)] A morphism $\overline{\alpha}$ in $(\Gr \cA) (a,b)$ consists of a composable string $(\overline{\alpha}_n,\hdots,\overline{\alpha}_1)$ of generator morphisms $\overline{\alpha_{k}}: \{f_i\}_{i=1}^{n_1} \to \{g_j\}_{j=1}^{n_2}$ which themselves consist of:
\begin{itemize}
    \item[(a)] Three numbers $k,\ell_1,\ell_2$ with $k \leq \ell_i$ such that
    \begin{itemize}
        \item[$\bullet$] If $m < k$, then $f_m = g_m$ and
        \item[$\bullet$] If $m > 0$, then $f_{\ell_1 + m} = g_{\ell_2 + m}$ if either side exists. (so $n_1 - \ell_1 = n_2 - \ell_2$)
    \end{itemize}
    \item[(b)] A pair $(\sigma,\tau)$ where $\sigma = (\sigma,D)$ and $\tau = (\tau,E)$ are so-called associations for $\{f_i\}_{i = k}^{\ell_1}$ and $\{g_j\}_{j = k}^{\ell_2}$ respectively. Since $\cA \in \GrayCat$, the composition $\boxtimes$ of 1-morphisms in $\cA$ is associative, so this data is superfluous and we will not present more detail.
    \item[(c)] A 2-morphism between the ``evaluated associations'' for $\{f_i\}_{i=k}^{\ell_1}$ and $\{g_j\}_{j=k}^{\ell_2}$. In our case, this amounts to a 2-morphisms $\alpha: f_{\ell_1} \boxtimes \cdots \boxtimes f_k \Rightarrow g_{\ell_2} \boxtimes \cdots \boxtimes g_k$ in $\cA$.
\end{itemize}
The composition $\otimes$ of $1$-morphisms in $(\Gr T)(a,b)$ is given by concatenation, so the empty 1-morphism $\varnothing_{\{f_i\}}$ is the identity on $\{f_i\}$.

For a basic 2-morphism $\overline{\alpha}$ and a 1-morphism $\{h_k\}_{k=1}^{n_3}$ in $\Gr \cA$ we define their composites $\overline{\alpha} \boxtimes \{h_k\}: \{h_k,f_i\} \Rightarrow \{h_k,g_j\}$ and $\{h_k\} \boxtimes \overline{\alpha}: \{f_i,h_k\} \Rightarrow \{g_j,h_k\}$ (whenever they make sense) by
\begin{align*}
\overline{\alpha} \boxtimes \{h_k\} &\coloneqq (k + n_3, \ell_1 + n_3, \ell_2 + n_3, \sigma, \tau, \alpha).\\
\{h_k\} \boxtimes \overline{\alpha} &\coloneqq (k, \ell_1, \ell_2, \sigma, \tau, \alpha).
\end{align*}

We extend $\boxtimes$ for arbitrary 2-morphisms $\overline{\alpha} = (\overline{\alpha}_n,\hdots,\overline{\alpha}_1)$ by
\begin{align*}
\overline{\alpha} \boxtimes \{h_k\} &\coloneqq (\overline{\alpha}_n \boxtimes \{h_k\} ,\hdots,\overline{\alpha}_1 \boxtimes \{h_k\} ),\\
\{h_k\} \boxtimes \overline{\alpha} &\coloneqq (\{h_k\} \boxtimes \overline{\alpha}_n,\hdots,\{h_k\} \boxtimes \overline{\alpha}_1 ).
\end{align*}
With the data presented thus far, $\Gr \cA$ forms a sesquicategory which is free on a computad. Once we finish constructing the $\Gray$-category $\Gr \cA$, \cite[Corollary 9.4]{lack_2011} will yield that $\Gr \cA$ is cofibrant.

\item[($\Gr$3)] For generator 1-morphisms $\overline{\alpha}, \overline{\beta}$ in $(\Gr \cA)(a,b)$, a 2-morphism $\Gamma: \overline{\alpha} \Rightarrow \overline{\beta}$ in $(\Gr \cA)(a,b)$ is simply a 3-morphism $\Gamma:  [\alpha] \Rrightarrow [\beta]$ in $\cA$ where $[\alpha],[\beta]: f_{n_1} \boxtimes \cdots \boxtimes f_1 \Rightarrow g_{n_2} \boxtimes \cdots \boxtimes g_1$ are given by 
\begin{align*}
[\alpha] \coloneqq f_{n_1} \boxtimes \cdots \boxtimes f_{\ell_1+1} \boxtimes \alpha \boxtimes f_{k-1} \boxtimes \cdots \boxtimes f_1,\\
[\beta] \coloneqq  g_{n_2} \boxtimes \cdots \boxtimes g_{\ell_2+1}  \boxtimes \beta \boxtimes g_{k-1} \boxtimes \cdots \boxtimes g_1.
\end{align*} 
For general 1-morphisms $\overline{\alpha} = \overline{\alpha}_n \otimes \cdots \otimes \overline{\alpha}_1$ and $\overline{\beta} = \overline{\beta}_{n'} \otimes \cdots \otimes \overline{\beta}_1$ in $\Gr \cA(a,b)_1$, $\Gamma: \overline{\alpha} \Rightarrow \overline{\beta}$ is simply a 3-morphism $\Gamma: [\alpha_n] \otimes \cdots \otimes [\alpha_1] \Rrightarrow [\beta_{n'}] \otimes \cdots \otimes [\beta_1]$ in $\cA$.

The vertical composition $\circ$ of 2-morphisms in $(\Gr \cA)(a,b)$ is inherited from $\cA$ and is thus strictly associative and unital. The horizontal composition $\otimes$ of 2-morphisms in $(\Gr \cA)(a,b)$ is also inherited from $T$, so this composition satisfies the interchange law and is strictly associative. Thus $\Gr T (a,b)$ is actually a strict 2-category.

For a 1-morphism $\{h_k\}$ in $\Gr \cA$ we define $\{h_k\} \boxtimes \Gamma: \{h_k\} \boxtimes \overline{\alpha} \Rrightarrow \{h_k\} \boxtimes \overline{\beta}$ (whenever it makes sense) by
$$(\boxtimes_k h_k) \boxtimes ([\alpha_n] \otimes \cdots \otimes [\alpha_1] ) \xRrightarrow{ (\boxtimes_k h_k) \boxtimes \Gamma} (\boxtimes_k h_k) \boxtimes ( [\beta_{n'}] \otimes \cdots \otimes [\beta_1]).$$
We similarly  define $\Gamma \boxtimes \{h_k\}$.

\item[($\Gr\!\Sigma$)] For basic $2$-morphisms $\overline{\alpha}$ and $\overline{\beta}$ in $\Gr \cA$ we may define the interchanger $\Sigma_{\overline{\alpha},\overline{\beta}}$ (whenever it makes sense) to be the interchanger $\Sigma_{[\alpha],[\beta]}$ in $\cA$. One then extends $\Sigma$ to general 2-morphisms in $\Gr \cA$ by the same formula in \hyperref[Interchanger:Composition]{(C4)}.
\end{itemize}
By \cite[Thm.~10.8]{gurski_2013}, this data serves to equip $\Gr \cA$ with the structure of a $\Gray$-category. We continue recalling \cite[Def.~10.7]{gurski_2013} and \cite[Thm.~10.9]{gurski_2013}:
\end{construction}

\begin{construction}
 We define the $\Gray$-functor $\ev_\cA: \Gr \cA \to \cA$ as
 \begin{itemize}
     \item[(ev0)] On objects, $\ev_\cA$ is the identity.
     \item[(ev1)] On $1$-morphisms,
     $$\ev(\{f_i\}) = [f_i] : = f_n \boxtimes \cdots \boxtimes f_1$$
     where $[\;\;] = I_a$ for the empty 1-cell $\varnothing_a: a \to a$. We will also use the notation $[b_j,a_i] \coloneqq  [b_j] \boxtimes [a_i]$.
    \item[(ev2)] 
    For a basic 2-morphism $\overline{\alpha}: \{f_i\} \Rightarrow \{g_j\}$, we define $\ev_\cA(\overline{\alpha}) = [\alpha]$ where we are using the same notation as in Construction \ref{construction:Gr}.
    We then extend $\ev_{\cA}$ to general 2-morphisms in $\cA$ by
    $$\ev(\overline{\alpha}_n \otimes \cdots \otimes \overline{\alpha}_1) \coloneqq \ev(\overline{\alpha}_n) \otimes \cdots \otimes \ev(\overline{\alpha}_1).$$
    
    and set $\ev(\varnothing_{\{f_i\}}) \coloneqq \id_{[f_i]}$.
    \item[(ev3)] On $3$-morphisms, $\ev_\cA$ is the identity.
\end{itemize}
\begin{remark}
By construction $\ev_\cA$ is surjective at all levels and (fully) faithful at the top level, i.e., $\ev_\cA$ is a trivial fibration in Lack's model structure for $\GrayCat$. So $\Gr \cA \xrightarrow{\ev_\cA} \cA$ is indeed the cofibrant replacement for $\cA$ in $\GrayCat$.
\end{remark}
\end{construction}
We now recall \cite[\S10.6]{gurski_2013}.
\begin{construction}
For a weak 3-functor $F: \cA \to \cB$ between $\Gray$-categories, we define the $\Gray$-functor $\Gr F: \Gr \cA \to \Gr \cB$ together with an equivalence pseudo-icon $(\varphi,M,\Pi)$ where

\begin{center}
\small
\begin{tikzcd}[column sep = 0pt, row sep = 10]
\Gr \cA \arrow[dd, "\ev_\cA"'] \arrow[rr, "\Gr F"] &                     & \Gr \cB \arrow[dd, "\ev_\cB"] \\
                                                                                                & \varphi\Downarrow\quad &                                             \\
\cA \arrow[rr, "F"']                                                     &                     & \cB                                      
\end{tikzcd}
\end{center}

\begin{itemize}
    \item[($\Gr$F0)] $\Gr F(s) = F(s)$ for every object $s \in S$.
    \item[($\Gr$F1)] $\Gr F(\{f_i\}) = \{F f_i\}$ for a string $\{f_i\} \in \Gr S$.
    
    \item[($\varphi1$)] Let $\{f_i\}_{i=1}^n$ be a 1-morphism in $\Gr \cA$. 
    If $n = 0$, so that $\{f_i\} = \varnothing_a$, we define the equivalence
    $\varphi_{\{f_i\}}: I_{F(a)} \Rightarrow F(\varnothing_a)$
    to be the unitor $\varphi_{\{f_i\}} \coloneqq  F^0_a$ of $F$.
    When $n = 1$, we define
    $\varphi_{\{f_i\}}: F(f_1) \Rightarrow F(f_1)$
    to be the identity $\varphi_{\{f_i\}} \coloneqq  \id_{F(f_1)}$. 
    When $n > 2$, we define
    $\varphi_{\{f_i\}}: Ff_n \boxtimes \cdots \boxtimes Ff_1 \Rightarrow F\big(f_n \boxtimes \cdots \boxtimes f_1\big)$
    to be the leftmost composition of tensorators $F^2$ of $F$ tensored with identities
    $$F^2_{f_n \boxtimes \cdots \boxtimes f_2 ,f_1} \otimes \cdots \otimes (F^2_{f_n \boxtimes f_{n-1},f_{n-2}} \boxtimes  Ff_{n-3} \boxtimes \cdots \boxtimes  Ff_1) \otimes (F^2_{f_n,f_{n-1}} \boxtimes Ff_{n-2} \boxtimes Ff_{n-3} \boxtimes \cdots \boxtimes Ff_1).$$
     We also choose the obvious adjoint $\varphi^{\sqdot}$, unit, and counit.
     For simplicity, we will continue to denote this $n$-fold tensor by $[\,\cdot\,]$, so that
    $$\varphi_{\{f_i\}}: [Ff_i] \Rightarrow F[f_i].$$
    
\item[($\Gr$F2)] For a basic $\overline{\alpha} = (k,\ell_1,\ell_2,(\sigma,D),(\tau,E),\alpha)$, we define
$$\Gr F(\overline{\alpha}) = (k,\ell_1,\ell_2,(\sigma,FD),(\tau,FE), \Gr F \alpha).$$
where we recall the data of $(\sigma,FD)$ and $(\tau,FE)$ is superfluous and define $\Gr F \alpha$ so that the following diagram commutes.
\vspace{-10pt}
\begin{center}
\small
\begin{tikzcd}
{[Ff_i]} \arrow[r, "\Gr F \alpha", Rightarrow] \arrow[d, "{[\varphi_{\{f_{\ell_1 \geq i \geq k}\}}]}"', Rightarrow] & {[F g_j]}                                                                                                                           \\
{[Ff_{i>\ell_1},F[f_{\ell_1 \geq i \geq k}],Ff_{k > i}] } \arrow[r, "{[F\alpha]}"', Rightarrow]                             & {[Fg_{j>\ell_2},F[g_{\ell_2 \geq j \geq k}],Fg_{k > j}] } \arrow[u, "{[\varphi^{\sqdot}_{\{g_{\ell_2 \geq j \geq k}\}}]}"', Rightarrow]
\end{tikzcd}
\end{center}
We extend this for general 2-morphisms by
$\Gr F (\overline{\alpha}_n \otimes \hdots \otimes \overline{\alpha}_1) \coloneqq \Gr F(\overline{\alpha}_n) \otimes \hdots \otimes \Gr F (\overline{\alpha}_1).$

\item[($\varphi2$)]\label{varphi2}
For a basic 2-morphism $\overline{\alpha}$ in $\Gr \cA$, we again denote $f_{n_1} \boxtimes \cdots \boxtimes f_{\ell_1} \boxtimes \alpha \boxtimes f_{k-1} \boxtimes \cdots \boxtimes f_1$ by $[\alpha]$.
We then define the naturality isomorphism $\varphi_{\overline{\alpha}}$ to be the following composite of adjoint equivalence data from $F$
\begin{center}
\small
\begin{tikzcd}[row sep = 7pt, column sep = 7pt]
{[Ff_i]} \arrow[ddd, "\varphi_{\{f_i\}}"', Rightarrow] \arrow[rrr, "{[\varphi_{\{f_{\ell_1 \geq i \geq k}\}}]}", Rightarrow] &       &       & {[Ff_{i>\ell_1},F[f_{\ell_1 \geq i \geq k}],Ff_{k > i}]} \arrow[lllddd, "\varphi_{\set{f_{i>\ell_1},[f_{\ell_1 \geq i \geq k}],f_{k>i}}}" description, Rightarrow] \arrow[ddd, "{[F\alpha]}", Rightarrow]                                       \\
                                                                                           & \qquad\cong\qquad\qquad\qquad &       &                                                                                                                                                     \\
                                                                                           &       &       &                                                                                                                                                     \\
{F[f_i]} \arrow[ddd, "{F[\alpha]}"', Rightarrow]                                           &       &       & {[Fg_{j>\ell_2},F[g_{\ell_2 \geq j \geq k}],Fg_{k > j}]} \arrow[lllddd, "\varphi_{\set{g_{j>\ell_2},[g_{\ell_2 \geq j \geq k}],g_{k>j}}}" description, Rightarrow] \arrow[ddd, "{[\varphi^{\sqdot}_{\{g_{\ell_2 \geq j \geq k}\}}]}", Rightarrow] \arrow[lll, "\qquad\cong", phantom] \\
                                                                                           &       &       &                                                                                                                                                     \\
                                                                                           &       & \qquad\qquad\qquad\qquad\cong &                                                                                                                                                     \\
{F[g_j]}                                                                                   &       &       & {[Fg_j]} \arrow[lll, "\varphi_{\{g_j\}}", Rightarrow]                                                                                            
\end{tikzcd}
\end{center}
We then extend this to general 2-morphisms in $\Gr \cA$ as usual.

\item[(M)] For every object $a \in \Gr \cA_0 = \cA_0$, we define the invertible 3-cell $M_a$ 
\begin{center}
\small
\begin{tikzcd}[column sep = 0pt, row sep = 10]
\id_{Fa} \arrow[dd, no head] \arrow[dd, no head, shift right] \arrow[rr, no head]  \arrow[rr, no head, shift right] &                     & \id_{Fa} \arrow[dd, "\varphi_{\varnothing_a}", Rightarrow] \\
                                                                                                & M_a\Ddownarrow\quad &                                             \\
\id_{Fa} \arrow[rr, "F^0_a"', Rightarrow]                                                     &                     & F (\varnothing_a)                                     
\end{tikzcd}
\end{center}
to be the identity $\id_{F^0_a}$. 
\item[($\Pi$)] We define the modification $\Pi$ to be the unique coherence isomorphism given by $F$.
In particular, $\Pi$ has component invertible 3-cells $\Pi_{\{f_i\},\{g_j\}}$ for $a \xrightarrow{\{f_i\}} b \xrightarrow{\{g_j\}} c$ in $\Gr \cA$
\begin{center}
\small
\begin{tikzcd}[column sep = 0pt, row sep = 10]
[Fg_j] \boxtimes [Ff_i] \arrow[dd, "\varphi_{\{g_j\}} \boxtimes \varphi_{\{f_i\}}"', Rightarrow]  \arrow[rr, no head] \arrow[rr, no head,shift right] &                     & \left[Ff_i, Fg_j\right] \arrow[dd, "\varphi_{\{g_j\} \boxtimes \{f_i\}}", Rightarrow] \\
                                                                                                & \Pi_{f,g}\Ddownarrow\qquad &                                             \\
F[g_j] \boxtimes F[f_i] \arrow[rr, "F^2_{[f_i],[g_j]}"', Rightarrow]                                                     &                     & F[f_i, g_j]                                    
\end{tikzcd}
\end{center}

\item[($\Gr$F3)] For a 3-morphism $\Gamma: \overline{\alpha} \Rrightarrow \overline{\beta}$ in $\Gr \cA$ we define $\Gr F (\Gamma): \Gr F (\overline{\alpha}) \Rrightarrow \Gr F (\overline{\beta})$ so that the following diagram commutes
\begin{center}
\small
\begin{tikzcd}
{[\varphi^{\sqdot}_{\{f_{\ell_2 \geq j \geq k}\}}] \otimes [F \alpha] \otimes [\varphi_{\{f_{\ell_1 \geq i \geq k\}}}] } \arrow[d, "\phi_{\,\overline{\alpha}}"'] \arrow[r, "\Gr F(\Gamma)\,"] & {[\varphi^{\sqdot}_{\{f_{\ell_2 \geq j \geq k}\}}] \otimes [F \beta] \otimes [\varphi_{\{f_{\ell_1 \geq i \geq k\}}}]} \\
{\varphi^{\sqdot}_{\{g_j\}} \otimes F[\alpha] \otimes \varphi_{\{f_i\}}} \arrow[r, "\id \otimes F\Gamma \otimes \id"']                                                                     & {\varphi^{\sqdot}_{\{g_j\}} \otimes F[\beta] \otimes \varphi_{\{f_i\}}} \arrow[u, "\phi^{\sqdot}_{\,\overline{\beta}}"']
\end{tikzcd}
\end{center}
where the 3-isomorphisms $\phi_{\overline{\alpha}}$ and $\phi^{\sqdot}_{\overline{\beta}}$ are similar to $\varphi_{\overline{\alpha}}$ and $\varphi^{\sqdot}_{\overline{\beta}}$ in \hyperref[varphi2]{($\varphi2$)}.
\end{itemize}
\end{construction}

\begin{properties}
We now review the properties outlined in Proposition \ref{fact:hat}.
\begin{enumerate}
\item[($\widehat{1}$)] \label{item:hat1} The fact that $(\varphi,M,\Pi)$ forms a pseudo-icon follows from the coherence theorem for weak 3-functors. We refer the interested reader to \cite[Thm.~10.13]{gurski_2013} and its corollary \cite[Cor.~10.15]{gurski_2013}.

\item[($\widehat{2}$)] \label{item:hat2} When $F$ is $\Gray$, $F$ is strict so
for every 1-morphism $\{f_i\}$ in $\Gr A$, $\varphi_{\{f_i\}} = \id$ by construction. Now consider a generator 2-morphism $\overline{\alpha} = (k,\ell_1,\ell_2,(\sigma,D),(\tau,E),\alpha)$ from $\{f_i\}$ to $\{g_j\}$,
$$\Gr F(\overline{\alpha}) = (k,\ell_1,\ell_2,(\sigma,FD),(\tau,FE),F\alpha).$$
Therefore, the following diagram strictly commutes
\begin{center}
\small
\begin{tikzcd}[column sep = 50pt]
\ev_\cB \circ \Gr F \{f_i\} \arrow[d, "\varphi_{\{f_i\}}"',Rightarrow] \arrow[r, "\ev_\cB \circ \Gr F (\overline{\alpha})",Rightarrow] & \ev_\cB \circ \Gr F \{g_j\} \arrow[d, "\varphi_{\{g_j\}}",Rightarrow] &  & {[Ff_i]} \arrow[d, no head] \arrow[d, no head, shift right] \arrow[r, "{[F\alpha]}",Rightarrow] & {[Fg_j]} \arrow[d, no head] \arrow[d, no head, shift left] \\
F \circ \ev_\cA \{f_i\} \arrow[r, "F \circ\, \ev_\cA (\overline{\alpha})"',Rightarrow]                                          & F \circ \ev_\cA \{g_j\}                                      &  & {F[f_i]} \arrow[r, "{F[\alpha]}"',Rightarrow]                                                   & {F[g_i]}                                                  
\end{tikzcd}
\end{center}
In other words, $\ev_\cB \circ \Gr F (\overline{\alpha}) = F \circ \ev_\cA (\overline{\alpha})$ and since both functors are strict, this equality holds for general 2-morphisms. Finally, since $\Gr F$ and $\ev_\cA$ are uniquely determined on 3-morphisms by coherence, it is immediate that $\ev_\cB \circ \Gr F = F \circ \ev_\cA$ at the level of 3-morphisms.
\item[($\widehat{3}$)] Since $\id_\cA$ is $\Gray$, the result is immediate from our discussion in \hyperref[item:hat2]{($\widehat{2}$)}.
\item[($\widehat{4}$)] From our discussion in \hyperref[item:hat1]{($\widehat{1}$)} we see that
\begin{center}
\small
\begin{tikzcd}[column sep = 0pt, row sep = 15pt]
                                                          &  &[10pt] \Gr \cA \arrow[rr, "\Gr F"] \arrow[dd, "\ev_\cA"'] &                      & \Gr \cB \arrow[dd, "\ev_\cB"] &                      &                                                    \\
                                                          &  &                                                    & \varphi^F \Downarrow &                               &                      &                                                    \\
                                                          &  & \cA \arrow[rr, "F"'] \arrow[rrdd, "GF"]            &                      & \cB \arrow[dd, "G"']          &                      & \Gr \cB \arrow[ll, "\ev_\cB"'] \arrow[dd, "\Gr G"] \\
                                                          &  &                                                    &                      &                               & \Leftarrow \varphi^G &                                                    \\
\Gr \cA\qquad \arrow[rruu, "\ev_\cA"] \arrow[rrdd, "\Gr (GF)"'] &  & \qquad \varphi^{GF}\Nearrow                               &                      & \cC                           &                      & \Gr \cC \arrow[ll, "\ev_\cC"]                      \\
                                                          &  &                                                    &                      &                               &                      &                                                    \\[19pt]
                                                          &  & \Gr \cC \arrow[rruu, "\ev_\cC"']                   &                      &                               &                      &                                                   
\end{tikzcd}
\end{center}
Therefore,
\begin{align*}
\Gr \cA \xrightarrow{\Gr (GF)} \Gr \cC 
&\cong \Gr \cA \xrightarrow{\ev_\cA} \cA \xrightarrow{F} \cB \xrightarrow{G} \cC \xrightarrow{\ev_\cC^{\sqdot}} \Gr \cC\\
&\cong \Gr \cA \xrightarrow{\ev_\cA} \cA \xrightarrow{\ev_\cA^{\sqdot}} \Gr \cA \xrightarrow{\Gr F} \Gr B \xrightarrow{\ev_\cB} \cB \xrightarrow{\ev_\cB^{\sqdot}} \Gr \cB \xrightarrow{\Gr G} \Gr \cC \xrightarrow{\ev_C} \cC \xrightarrow{\ev_C^{\sqdot}} \Gr \cC\\
&\cong \Gr \cA \xrightarrow{\Gr F} \Gr \cB \xrightarrow{\Gr G} \Gr \cC ,
\end{align*}
where $\ev_\cA^{\sqdot},\ev_\cB^{\sqdot},\ev_\cC^{\sqdot}$ are pseudo-inverse to $\ev_\cA,\ev_\cB,\ev_\cC$ respectively.

\item[($\widehat{5}$)] Notice that $\Gr (G \circ F)$ and $\Gr G \circ \Gr F$ always agree on objects and 1-morphisms in $\Gr \cA$.
We now suppose $F$ is $\Gray$. Consider a generator 2-morphism $\overline{\alpha} = (k,\ell_1,\ell_2,(\sigma,F),(\tau,E),\alpha)$ in $\Gr \cA$. 
We must show $\Gr G \Gr F \alpha = \Gr (G \circ F) \alpha$. 
As we saw in ($\widehat{2}$), $\Gr F(\overline{\alpha}) = (k,\ell_1,\ell_2,(\sigma,FD),(\tau,FE),F\alpha)$ and
$$\Gr G \Gr F \alpha = (\varphi^G)^{\sqdot}_{\{Fg_{\ell_2 \geq j \geq k}\}} \otimes GF\alpha \otimes \varphi^G_{\{Ff_{\ell_1 \geq i \geq k}\}}.$$
This agrees with 
$$\Gr (G \circ F)(\overline{\alpha}) =  (\varphi^{GF})^{\sqdot}_{\{g_{\ell_2 \geq j \geq k}\}} \otimes GF\alpha \otimes \varphi^{GF}_{\{f_{\ell_1 \geq i \geq k}\}}.$$
because in this case the constraint data for $GF$ is simply that of $G$ restricted to the image of $F$. 
Since $\Gr G \circ \Gr F$ and $\Gr( G \circ F)$ are strict, we conclude that these functors agree on general 2-morphisms in $\Gr \cA$. 
We conclude by a similar argument that these functors agree on 3-morphisms as well.
Thus $\Gr(G \circ F) = \Gr G \circ \Gr F$.

Now suppose $G$ is $\Gray$. We must again show $\Gr F \Gr G \alpha = \Gr (G \circ F) \alpha$ for a generator 2-morphism $\overline{\alpha}$ in $\Gr \cA$. Recall that
$$\Gr F \alpha = (\varphi^F)^{\sqdot}_{\{g_{\ell_2 \geq j \geq k}\}} \otimes F\alpha \otimes \varphi^F_{\{f_{\ell_1 \geq i \geq k}\}}.$$
Thus
$$\Gr G \circ \Gr F \alpha = G(\varphi^F)^{\sqdot}_{\{g_{\ell_2 \geq j \geq k}\}} \otimes GF\alpha \otimes G\varphi^F_{\{f_{\ell_1 \geq i \geq k}\}}.$$
This agrees with
$$\Gr (G \circ F) \alpha =  (\varphi^{GF})^{\sqdot}_{\{g_{\ell_2 \geq j \geq k}\}} \otimes GF\alpha \otimes \varphi^{GF}_{\{f_{\ell_1 \geq i \geq k}\}}.$$
because in this case, the constraint data for $GF$ is that of $F$ pushed through the functor $G$.
We then conclude that $\Gr (G \circ F)$ and $\Gr G \circ \Gr F$ agree on arbitrary 2-morphisms in $\Gr \cA$ as usual. A similar argument reveals that these functors also agree on 3-morphisms, and we again conclude $\Gr (G \circ F) = \Gr G \circ \Gr F$.
\end{enumerate}
\end{properties}

\section{Lack's path object construction}\label{appendix:PB}
To show \hyperref[fact:BI]{(P1)}, we recall \cite[Prop.~4.1]{lack_2011} and include details which were originally left to the reader.
\begin{construction}
The path $\Gray$-category $\bbP \cB$ of a $\Gray$-category $\cB$ is constructed as follows.
\begin{itemize}
\item[($\bbP0$)] An object $a \in \bbP \cB$ is a biequivalence $\Vec{a}: Sa \to Ta$ in $\cB$.
\item[($\bbP1$)] A 1-morphism $f: a \to b$ in $\bbP \cB$ consists a tuple $f = (Sf,Tf,\Vec{f}\,)$ where $Sf: Sa \to Sb$ and $Tf: Ta \to Tb$ are 1-morphisms in $\cB$, and $\Vec{f}$ is an equivalence in $\cB$ with:
\begin{center}
\small
\begin{tikzcd}[column sep = 5pt, row sep = 5pt]
Sa \arrow[dd, "\Vec{a}"'] \arrow[rr, "Sf"] &              & Sb \arrow[dd, "\Vec{b}"] \\
                                 & \Downarrow \Vec{f} &                         \\
Ta \arrow[rr, "Tf"']                 &              & Tb                  
\end{tikzcd}
\end{center}
\item[($\bbP2$)] A 2-morphism $\theta: f \Rightarrow g$ in $\bbP \cB$ consists of a tuple $(S\theta,T\theta,\Vec{\theta}\,)$ where $S\theta: Sf \Rightarrow Sg$ and $T\theta: Tf \Rightarrow Tg$ are 2-morphisms in $\cB$, and $\Vec{\theta}$ is an invertible 3-morphisms in $\cB$ with:
\begin{center}
\small
\begin{tikzcd}[column sep = 5pt, row sep = 2pt]
Sa \arrow[dd, "\Vec{a}" description] \arrow[rr, "Sg"', ""{name=Dr,inner sep=2pt}] \arrow[rr, "Sf", bend left, shift left=3, ""{name=Ur,inner sep=2pt,below}] &                     & Sb \arrow[dd, "\Vec{b}" description] &                              &   Sa \arrow[dd, "\Vec{a}" description] \arrow[rr, "Sf"]            &                     & Sb \arrow[dd, "b_\alpha" description]\\
                                                                    & \Downarrow \Vec{g} &                                        & \xRrightarrow{\Vec{\theta}} &                                                                                                      & \Downarrow \Vec{f} &                                          \\
Ta \arrow[rr, "Tg"']                                                                               &                     & Tb               &                              &     Ta \arrow[rr, "Tf", ""{name=U,inner sep=2pt,below}] \arrow[rr, "Tg"', bend right, shift right=3, ""{name=D,inner sep=2pt}] &                     & Tb                                                      
\arrow[Rightarrow, from=U, to=D, "\;T\theta\;"]
\arrow[Rightarrow, from=Ur, to=Dr, "\;S\theta\;"]
\end{tikzcd}
\end{center}
\item[($\bbP3$)] A 3-morphism $\Gamma: \theta \Rrightarrow \sigma$ in $\bbP \cB$ consists of a tuple $(S\Gamma, T\Gamma)$ where $S\Gamma: S\theta \Rrightarrow S\sigma$ and $T\Gamma: T\theta \Rrightarrow T\sigma$ are 3-morphisms in $\cB$ that commute with $\Vec{\theta}$ and $\Vec{\sigma}$ in the obvious way.
\end{itemize}
There is an organic way of equipping $\bbP \cB$ with the structure of a $\Gray$-category by inheriting composites and interchangers from the $\Gray$-category $\cB$. We now define the ``source'' and ``target'' $\Gray$-functors $S,T: \bbP\cB \to \cB$.
\begin{itemize}
    \item[(ST0)] For an object $a \in \bbP \cB_0$, we set $S(a) \coloneqq Sa$ and $T(a) \coloneqq Ta$,
    \item[(ST1)] For a 1-morphism $f \in \bbP \cB_1$, we set $S(f) \coloneqq Sf$ and $T(f) \coloneqq Tf$,
    \item[(ST2)] For a 2-morphism $\theta \in \bbP \cB_2$, we set $S(\theta) \coloneqq S\theta$ and $T(\theta) \coloneqq T\theta$,
    \item[(ST3)] For a 3-morphism $\Gamma \in \bbP \cB_3$, we set $S(\Gamma) \coloneqq S\Gamma$ and $T(\Gamma) \coloneqq T\Gamma$
\end{itemize}
The fact that $S$ and $T$ are $\Gray$-functors is due to the fact that composites in $\bbP \cB$ are inherited from those in $\cB$. We now define the ``constant'' $\Gray$-functor $C: \cB \to \bbP\cB$. 
\begin{enumerate}
\item[(C0)] For an object $b \in \cB_0$, we set $C(b)$ to be the trivial biequivalence of $b$ with itself, i.e.  $C(b) \coloneqq \id_b$,
\item[(C1)] For a 1-morphism $f \in \cB_1$, we set
$C(f) \coloneqq (f,f,\id_{f})$,
\item[(C2)] For a 2-morphism $\theta \in \cB_2$, we set
$C(\theta) \coloneqq (\theta,\theta,\id_\theta)$,
\item[(C3)] For a 3-morphism $\Gamma \in \cB_3$, we set
$C(\Gamma) \coloneqq (\Gamma,\Gamma).$
\end{enumerate}    

The fact that $C: \cB \to \bbP\cB$ is a $\Gray$-functor is also due to the fact that composites in $\bbP \cB$ are inherited from those in $\cB$.
Furthermore, we have $S \circ C = T \circ C = \id_{\cB}$ by construction.
    
To prove $C$ is a weak equivalence, we must show $C$ is fully faithful and (weakly) essentially surjective. 
First, it is clear that $C$ is faithful at the level of 3-morphisms. 
To see that $C$ is full at this level, consider some $(S\Gamma,T\Gamma): C(\theta) \Rrightarrow C(\sigma)$ in $\bbP \cB$, where $\theta,\sigma: f \Rightarrow g$ and $f,g: a \to b$ in $\cB$.
Then $S\Gamma,T\Gamma: \theta \Rrightarrow \sigma$ in $\cB$ and the axiom for 3-morphisms in $\bbP\cB$ yields
$$S\Gamma = \id_\sigma \circ (\id_{\id_g} \otimes (b \boxtimes S\Gamma)) = ((T\Gamma \boxtimes a) \otimes \id_{\id_f}) \circ \id_\theta = T\Gamma.$$
Thus $\Gamma = C(S\Gamma) = C(T\Gamma)$, so $C$ is full on 3-morphisms.
To see $C$ is essentially surjective on $\bbP \cB(C(f),C(g))$ for $f,g \in \cB$, consider some $\theta: C(f) \Rightarrow C(g)$ in $\bbP \cB$.
We claim that $\theta \cong C(S\theta)$.
Indeed, in this case we may set $T\Gamma := \theta_\alpha^{-1}$ so that $T\Gamma: S\theta \Rrightarrow T\theta$ is a 3-isomorphism in $\cB$.
It is then clear that $(\id_{S\theta},T\Gamma): \theta \Rrightarrow C(S\theta)$ is a 3-isomorphism in $\bbP \cB$ and we conclude $C$ is fully faithful on each $\bbP \cB(C(b),C(b'))$. 
To verify $C$ is essentially surjective on $\bbP \cB(C(b),C(b'))$ for $b,b' \in \cB$, consider some $f: C(b) \to C(b')$ in $\bbP \cB$.
We claim that $f \simeq C(Sf)$.
First we choose an adjoint equivalence 
$$(\Vec{f}: Sf \Rightarrow Tf,\; \Vec{f}^{\,\sqdot}: Tf \Rightarrow Sf,\; \epsilon_{\Vec{f}}: \id_{Tf} \xRrightarrow{\sim} \Vec{f} \otimes \Vec{f}^{\,\sqdot},\; \eta_{\Vec{f}}: \Vec{f}^{\,\sqdot} \otimes \Vec{f} \xRrightarrow{\sim} \id_{Sf}).$$
We define $\theta: f \Rightarrow CSf$ and $\theta^{\sqdot}: CSf \Rightarrow f$ 
by $\theta = (\id_{Sf}, \Vec{f}^{\,\sqdot}, \eta_{\Vec{f}}^{-1})$ and $\theta^{\sqdot} = ( \id_{Sf}, \Vec{f}, \id_{\Vec{f}}).$ We then define the unit and counit $\epsilon_\theta: \id_{CSf} \Rightarrow \theta \otimes \theta^{\sqdot}$ and $\eta_\theta: \theta^{\sqdot} \otimes \theta \Rightarrow \id_f $ by $\epsilon_\theta =  (\id_{\id_{\Vec{f}}}, \eta^{-1}_{\Vec{f}}) $ and $\eta_\theta = (\id_{\id_{\Vec{f}}}, \epsilon^{-1}_{\Vec{f}})$. The fact that $\epsilon_\theta$ and $\eta_\theta$ are 3-isomorphisms in $\bbP \cB$ follows trivially for $\epsilon_\theta$ and from the zig-zag relations of an adjoint equivalence for $\eta_\theta$.
Therefore $C$ is fully faithful. 

We will now verify $C$ is (weakly) essentially surjective on $\cB$. Consider some object $a$ in $\bbP \cB$. 
We claim that $a$ is biequivalent to $CSa = \id_{Sa}$ in $\bbP \cB$.
We first choose a biadjoint biequivalence
$$(\Vec{a}: Sa \to Ta,\; \Vec{a}^{\,\sqdot}: Ta \to Sa,\; \mathbf{a_m},\; \mathbf{a_n},\; a_\Phi,\; a_\Psi),$$
where $\ba_\bm = (a_m,a_m^{\sqdot},\epsilon_{a_m},\eta_{a_m})$ and $\ba_\bn = (a_n,a_n^{\sqdot},\epsilon_{a_n},\eta_{a_n})$ are adjoint equivalences in $\cB(Sa,Sa)$ and in $\cB(Ta,Ta)$ respectively with
\begin{align*}
    a_m: \id_{Sa} \Rightarrow \Vec{a}^{\,\sqdot} \Vec{a} && a_n: \id_{Ta} \Rightarrow \Vec{a} \Vec{a}^{\,\sqdot}\\
    a_m^{\sqdot}: \Vec{a}^{\,\sqdot} \Vec{a} \Rightarrow \id_{Sa} && a_n^{\sqdot}: \Vec{a} \Vec{a}^{\,\sqdot} \Rightarrow \id_{Ta}\\
    \epsilon_{a_m}: \id_{\id_{Sa}} \xRrightarrow{\sim} a_m^{\sqdot} a_m && 
    \epsilon_{a_n}: \id_{\id_{Ta}} \xRrightarrow{\sim} a^{\sqdot}_n a_n \\
    \eta_{a_m}: a_m a_m^{\sqdot} \xRrightarrow{\sim} \id_{\Vec{a}^{\,\sqdot} \Vec{a}} && 
    \eta_{a_n}: a_n a_n^{\sqdot} \xRrightarrow{\sim} \id_{\Vec{a} \Vec{a}^{\,\sqdot}}
\end{align*}
and $a_\Phi,a_\Psi$ are 3-isomorphisms in $\cB$ with
\begin{align*}
a_\Phi&: (a_n^{\sqdot} \boxtimes \Vec{a})(\Vec{a} \boxtimes a_m) \xRrightarrow{\sim} \id_{{\Vec{a}}} & a_\Psi&: (\Vec{a}^{\,\sqdot} \boxtimes a_n^{\sqdot})(a_m \boxtimes \Vec{a}^{\,\sqdot}) \xRrightarrow{\sim} \id_{{\Vec{a}^{\,\sqdot}}}.
\end{align*}
Consider the following biequivalence in $\bbP \cB$.
\begin{itemize}
\item The pair of 1-morphisms  
\begin{tikzcd}
CSa \arrow[r, "f", bend left] & a\phantom{SC} \arrow[l, "g", bend left]
\end{tikzcd}
\hspace{-15pt} where $f \coloneqq (\id_{Sa}, \Vec{a}, \id_{\Vec{a}} )$ and $g \coloneqq  (\id_{Sa}, \Vec{a}^{\,\sqdot},a_m)$.
\item Equivalences $(\mu,\mu^{\sqdot},\epsilon_\mu,\eta_\mu)$ in $\bbP \cB(CSa,CSa)$ and $(\nu,\nu^{\sqdot},\epsilon_\nu,\epsilon_\nu)$ in $\bbP \cB(a,a)$ where
\begin{align*}
\mu&: \id_{CSa} \Rightarrow g \boxtimes f & \mu^{\sqdot}&: g \boxtimes f \Rightarrow \id_{CSa}\\
\nu&: \id_a \Rightarrow f \boxtimes g & \nu^{\sqdot}&: f \boxtimes g \Rightarrow \id_a\\
\epsilon_\mu&: \id_{\id_{CSa}} \xRrightarrow{\sim} \mu^{\sqdot} \otimes \mu & \eta_\mu&: \mu \otimes \mu^{\sqdot} \xRrightarrow{\sim} \id_{g \boxtimes f}\\
\epsilon_\nu&: \id_{\id_a} \xRrightarrow{\sim} \nu^{\sqdot} \otimes \nu & \eta_\nu &: \nu \otimes \nu^{\sqdot} \xRrightarrow{\sim} \id_{f \boxtimes g}
\end{align*}   
are given by
\begin{align*}
    \mu &\coloneqq  (\id_{\id_{Sa}}, a_m, \id_{a_m} ) & \mu^{\sqdot} &\coloneqq  (\id_{\id_{Sa}}, a^{\sqdot}_m, \epsilon_{a_m} )\\
    \nu &\coloneqq  (\id_{\id_{Sa}}, a_n, \nu_\alpha ) & \nu^{\sqdot} &\coloneqq  (\id_{\id_{Sa}}, a^{\sqdot}_n, a^{-1}_{\Phi}  )\\
\epsilon_\mu&\coloneqq  (\id_{\id_{\id_{Sa}}}, \epsilon_{a_m}) & \eta_\mu&\coloneqq  (\id_{\id_{\id_{Sa}}}, \eta_{a_m})\\
\epsilon_\nu&\coloneqq  (\id_{\id_{\id_{Sa}}}, \epsilon_{a_n}) & \eta_\nu &\coloneqq  (\id_{\id_{\id_{Sa}}}, \eta_{a_n})
\end{align*}
Here, the 3-isomorphism
$\nu_{\alpha}: \Vec{a} \boxtimes a_m \xRrightarrow{\sim} a_n \boxtimes \Vec{a}$
is defined as follows:
\begin{align*}
\nu_{\alpha} &\coloneqq  (\id_{a_n \boxtimes \Vec{a}} \otimes a_\Phi) \circ ((\eta_{a_n}^{-1} \boxtimes \Vec{a}) \otimes \id_{\Vec{a} \boxtimes a_m}),
\end{align*}
\end{itemize}   
The fact that $\epsilon_\mu,\eta_\mu,\epsilon_\nu,$ and $\eta_\nu$ are 3-isomorphisms in $\bbP \cB$ follows:
\begin{itemize}
    \item[($\epsilon_\mu$)] immediately,
    \item[($\eta_\mu$)] by the adjoint equivalence $\ba_\bm$'s zig-zag axioms,
    \item[($\epsilon_\nu$)] by the adjoint equivalence $\ba_\bn$'s zig-zag axioms,
    \item[($\eta_\nu$)] immediately.
\end{itemize}
Therefore $C$ is essentially surjective and we conclude $C: \cB \to \bbP \cB$ is a weak equivalence. The fact that $\binom{S}{T}: \cB^I \to \cB \times \cB$ is a fibration follows from a simple characterization of isomorphisms, equivalences, and bi-equivalences in $\cB^I$. As we have not provided much treatment for fibrations in Lack's model structure, we refer the interested reader to \cite{lack_2011} for more details.

\end{construction}

\begin{remark}
It is easy to show \hyperref[fact:BI]{(P2)} from our previous construction. Indeed, for a $\Gray$-functor $F: \cB_1 \to \cB_2$, notice there is an organic $\Gray$-functor $\bbP F: \bbP \cB_1 \to \bbP \cB_2$ which is obtained by passing the data of a $k$-morphism $(0 \leq k \leq 3$) in $\bbP \cB_1$ through $F.$
\end{remark}
\begin{remark}\label{rem:data}
When $F,G: \cA \to \cB$ are pseudo-naturally equivalent $\Gray$-functors, there exists a tritransformation $\alpha: F \to G$ consisting of:
\begin{itemize}
    \item[($\alpha0$)] for each object $a \in \cA$, a biequivalence $\alpha_a: F a \to G a$ and a 3-isomorphism $M_a: \alpha_{\id_a} \Rrightarrow \id_{\alpha_a}$ in $\cB$, where
    \item[($\alpha1$)] for each 1-morphism $f \in \cA$ these is an equivalence $\alpha_f: \alpha_b \boxtimes F f \Rightarrow G f \boxtimes \alpha_a$  in $\cB$ and for each composable pair $f,g$ of 1-morphisms, a 3-isomorphism $\Pi_{gf}: (G g \boxtimes \alpha_f)(\alpha_g \boxtimes F f) \Rrightarrow \alpha_{g \boxtimes f}$ in $\cB$, and
    \item[($\alpha2$)] for each $2$-morphism $\theta: f \Rightarrow g$ with $f,g: a \to b$, a 3-isomorphism $\alpha_\theta: \alpha_g \otimes (\alpha_b \boxtimes F \theta) \Rrightarrow (G \theta \boxtimes \alpha_a) \otimes \alpha_f$ in $\cB$.
\end{itemize}
This data is subject to various axioms, including naturality conditions for $\alpha$, $M$, and $\Pi$; an associativity condition $\Pi$; and left and right unitality conditions relating $M$ and $\Pi$. We refer the interested reader to \cite[Def.~4.16]{gurski_2013} for more details. We now prove \hyperref[fact:BI]{(P3)}.
\end{remark}
\begin{construction} Suppose $F,G: \cA \to \cB$ are pseudo-naturally equivalent $\Gray$-functors, so that we have data as in Remark \ref{rem:data}. We define the weak 3-functor $\langle F, G \rangle: \cA \to \bbP \cB$ as follows.
    \begin{enumerate}
        \item[(0)] For objects $a \in \cA$, we set
        $\langle F,G \rangle(a) \coloneqq (F a \xrightarrow{\alpha_a}  G a),$
        \item[(1)] For 1-morphisms $f \in \cA$, we set
        $\langle F, G \rangle(f) \coloneqq (F f,F g, \alpha_f),$ 
        \item[(2)] For 2-morphisms $\theta \in \cA$, we set
        $\langle F, G \rangle (\theta) \coloneqq (F \theta,G \theta, \alpha_\theta),$
        \item[(3)] For 3-morphism $\Gamma \in \cA$, we set
        $\langle F, G \rangle (\Gamma) \coloneqq (F \Gamma, G \Gamma).$
    \end{enumerate}
    The fact that $\langle F,G \rangle(\Gamma)$ is a 3-morphism in $\bbP \cB$ is by the naturality axiom $\alpha$ satisfies. We now equip this map with the structure of a weak 3-functor. 
    \begin{itemize}
        \item[($\chi$)] For $a \xrightarrow{f} b \xrightarrow{g} c$ in $\cA$, we define an adjoint equivalence $(\chi_{gf},\chi_{gf}^{\sqdot},\epsilon_{\chi_{gf}},\eta_{\chi_{gf}})$ where 
        $$\chi_{gf}: \langle F, G \rangle(g) \boxtimes \langle F,G \rangle (f) \Rrightarrow \langle F, G \rangle ( g \boxtimes f),$$
        by $\chi_{gf} \coloneqq (\id,\id, \Pi_{gf})$, $\chi_{gf}^{\sqdot} \coloneqq (\id,\id, \Pi_{gf}^{-1})$, and $\epsilon_{\chi_{gf}} = \eta_{\chi_{gf}} \coloneqq (\id,\id)$.
         Notice $\epsilon_{\chi_{gf}}$ and $\eta_{\chi_{gf}}$ immediately satisfy the axioms for 3-morphisms in $\bbP \cB$. 
        
        For \!\begin{tikzcd}
        a \arrow[r, "f", bend left,""{name=U,inner sep=1pt,below}] \arrow[r, "f'"', bend right,""{name=D,inner sep=1pt}] & b \arrow[r, "g", bend left,""{name=Uprime,inner sep=1pt,below}] \arrow[r, "g'"', bend right,""{name=Dprime,inner sep=1pt}] & c
        \arrow[Rightarrow, from=U, to=D, "\,\theta"]
        \arrow[Rightarrow, from=Uprime, to=Dprime, "\,\sigma"]
        \end{tikzcd}
        in $\cA$, we define the invertible 3-morphisms
        \vspace{-10pt}
        \begin{center}
        \small
        \begin{tikzcd}
        {\langle F, G\rangle(g) \boxtimes \langle F, G\rangle(f)} \arrow[r, "\chi_{gf}",Rightarrow] \arrow[d, "{\langle F, G\rangle(\sigma)\boxtimes \langle F, G\rangle(\theta)}"',Rightarrow] & {\langle F, G\rangle(g \boxtimes f)} \arrow[d, "{\langle F, G\rangle(\sigma \boxtimes \theta)}",Rightarrow] \\
        {\langle F, G\rangle(g') \boxtimes \langle F, G\rangle(f')} \arrow[r, "\chi_{g'f'}"',Rightarrow] \arrow[ru, "\Uuparrow\chi_{\sigma\theta}" description, phantom]                        & {\langle F, G\rangle(g' \boxtimes f')}                                                            
        \end{tikzcd}
        \hspace{10pt}\begin{tikzcd}
        {\langle F, G\rangle(g \boxtimes f)} \arrow[r, "\chi_{gf}^{\sqdot}",Rightarrow] \arrow[d, "{\langle F, G\rangle(\sigma \boxtimes \theta)}"',Rightarrow] & {\langle F, G\rangle(g) \boxtimes \langle F, G\rangle(f)} \arrow[d, "{\langle F, G\rangle(\sigma)\boxtimes \langle F, G\rangle(\theta)}",Rightarrow] \\
        {\langle F, G\rangle(g' \boxtimes f')} \arrow[r, "\chi_{g'f'}^{\sqdot}"',Rightarrow] \arrow[ru, "\Uuparrow\chi_{\sigma\theta}^{\sqdot}" description, phantom] & {\langle F, G\rangle(g') \boxtimes \langle F, G\rangle(f')}                                                                                  
        \end{tikzcd}
        \end{center}
simply by $\chi_{\sigma\theta} = \chi_{\sigma\theta}^{\sqdot} \coloneqq (\id,\id).$ The fact that $\chi_{\sigma\theta}$ and $\chi_{\sigma\theta}^{\sqdot}$ satisfy the axioms for 3-morphisms in $\bbP \cB$ follows from the naturality of $\Pi$.

        \item[($\iota$)] For each object $a \in \cA$, we define an adjoint equivalence $(\iota_a,\iota_a^{\sqdot},\epsilon_{\iota_a},\eta_{\iota_a})$ where
        $$\iota_a: \id_{\langle  F, G \rangle(a)} \Rightarrow \langle  F, G \rangle (\id_a),$$
        by $\iota_a \coloneqq (\id,\id,M_a^{-1})$, $\iota_a^{\sqdot} \coloneqq (\id,\id,M_a)$, and $\epsilon_{\iota_a} = \eta_{\iota_a} \coloneqq (\id,\id)$. Notice $\epsilon_{\iota_a}$ and $\eta_{\iota_a}$ immediately satisfy the axioms for 3-morphisms in $\bbP \cB$. We must also choose invertible 3-morphisms
\begin{align*}
\iota_{\id_a}&: \iota_a \otimes \id_{\id_{\langle F, G\rangle(a)}} \xRrightarrow{\sim} \langle F, G\rangle(\id_{\id_a}) \otimes \iota_a\\
\iota_{\id_a}^{\sqdot}&: \iota^{\sqdot}_a \otimes \langle F, G\rangle(\id_{\id_a}) \xRrightarrow{\sim} \id_{\id_{\langle F, G\rangle a}} \otimes \iota^{\sqdot}_a
\end{align*}
which we will simply take to be $\iota_{\id_a} \coloneqq (\id,\id)$ and $\iota^{\sqdot}_{\id_a} \coloneqq (\id,\id)$. The fact that $\iota_{\id_a}$ and $\iota^{\sqdot}_{\id_a}$ are 3-isomorphisms follows from the naturality condition $M$ satisfies.

        \item[($\omega$)] For $a \xrightarrow{f} b \xrightarrow{g} c \xrightarrow{h} d$ in $\cA$, we define the invertible 3-morphism
\begin{center}
\small
\begin{tikzcd}[column sep = 60pt]
{\langle  F,  G \rangle(h) \boxtimes \langle  F,  G \rangle(g) \boxtimes \langle  F,  G \rangle(f)} \arrow[r, "{\chi_{hg}\boxtimes \langle  F,  G \rangle(f)}",Rightarrow] \arrow[d, "{\langle  F,  G \rangle(h) \boxtimes \chi_{gf}}"',Rightarrow] & {\langle  F,  G \rangle(h \boxtimes g) \boxtimes \langle  F,  G \rangle(f)} \arrow[d, "{\chi_{hg,f}}",Rightarrow] \arrow[ld, "\Ddownarrow\omega_{hgf}" description, phantom] \\
{\langle  F,  G \rangle(h) \boxtimes \langle  F,  G \rangle(g \boxtimes f)} \arrow[r, "{\chi_{h,gf}}"',Rightarrow]                                                                                                                             & {\langle  F,  G \rangle(h\boxtimes g \boxtimes f)}                                                                                                         
\end{tikzcd}
\end{center}

simply by $\omega_{hgf} = (\id,\id)$. In this case, the axiom for 3-morphisms in $\bbP \cB$ translates to the associativity of $\Pi$.
        \item[($\gamma$)] For $a \xrightarrow{f} b$ in $\cA$, we define the invertible 3-morphism
        \begin{center}
\small
\begin{tikzcd}
{\id_{\langle F, G\rangle(b)} \boxtimes \langle F, G\rangle(f)} \arrow[rr, "{\iota_b \boxtimes\langle F ,  G \rangle (f)}",Rightarrow] \arrow[rd, no head] \arrow[rd, no head, shift right] & {} \arrow[d, "\Ddownarrow \gamma_f", phantom] & {\langle F, G\rangle(\id_b) \boxtimes \langle F, G\rangle(f)} \arrow[ld, "{\chi_{\id_b,f}}",Rightarrow] \\
                                                                                                                                                             & {\langle F, G\rangle(f)}                       &                                                                                                 
\end{tikzcd}
        \end{center}
        simply by $\gamma_f = (\id,\id)$.
        In this case, the axiom for 3-morphisms in $\bbP \cB$ translates to the left unitality axiom relating $\Pi$ and $M$.
        \item[($\delta$)] For $a \xrightarrow{f} b$ in $\cA$, we define the invertible 3-morphism $\delta_f: \id_{\langle  F, G \rangle(f)} \xRrightarrow{\sim} \chi_{f,\id_a}$
        \begin{center}
\small
\begin{tikzcd}
{\langle F, G\rangle(f) \boxtimes \id_{\langle F, G\rangle(a)}} \arrow[rd, no head] \arrow[rd, no head, shift right] \arrow[rr, "{{\langle F, G\rangle(f) \boxtimes\, \iota_a}}", Rightarrow] & {} \arrow[d, "\Uuparrow \delta_f", phantom] & {\langle F, G\rangle(f) \boxtimes \langle F, G\rangle(\id_a)} \arrow[ld, "{\chi_{f,\id_a}}", Rightarrow] \\
& {\langle F, G\rangle(f)}                  &                                                                                                            
\end{tikzcd}
        \end{center}
        by $\delta_f = (\id,\id)$. In this case, the axiom for 3-morphisms in $\bbP \cB$ translates to the right unitality axiom relating $\Pi$ and $M$.
    \end{itemize}
These collections of morphisms in $\bbP \cB$ assemble themselves into pseudonatural transformations or modifications due to the naturality conditions $\Pi$ and $M$ satisfy. 
Furthermore, these pseudonatural transformations and modifications trivially satisfy the two axioms of weak 3-functors. Thus $\langle  F, G \rangle$ is a weak 3-functor and it is clear that $S \circ \langle  F,  G \rangle =  F$ and $T \circ \langle  F,  G \rangle =  G$ by construction. 

\end{construction}

\bibliographystyle{amsalpha}
{\footnotesize{
\bibliography{refs}
}}
\end{document}